\documentclass{article}

\usepackage{authblk}
\usepackage[utf8]{inputenc}
\usepackage[dvipdfm,a4paper,left=20mm,right=20mm,top=20mm,bottom=20mm]{geometry}

\usepackage[mathlines]{lineno} 
\usepackage{amsmath}

\usepackage{etoolbox} 

\newcommand*\linenomathpatch[1]{%
  \cspreto{#1}{\linenomath}%
  \cspreto{#1*}{\linenomath}%
  \csappto{end#1}{\endlinenomath}%
  \csappto{end#1*}{\endlinenomath}%
}
\newcommand*\linenomathpatchAMS[1]{%
  \cspreto{#1}{\linenomathAMS}%
  \cspreto{#1*}{\linenomathAMS}%
  \csappto{end#1}{\endlinenomath}%
  \csappto{end#1*}{\endlinenomath}%
}

\expandafter\ifx\linenomath\linenomathWithnumbers
  \let\linenomathAMS\linenomathWithnumbers
  \patchcmd\linenomathAMS{\advance\postdisplaypenalty\linenopenalty}{}{}{}
\else
  \let\linenomathAMS\linenomathNonumbers
\fi

\linenomathpatch{equation}
\linenomathpatchAMS{gather}
\linenomathpatchAMS{multline}
\linenomathpatchAMS{align}
\linenomathpatchAMS{alignat}
\linenomathpatchAMS{flalign}

\usepackage{amssymb}
\usepackage{amsthm}

\usepackage{mathrsfs}
\usepackage{graphicx}
\usepackage{hyperref}
\usepackage{subfigure}
\usepackage{todonotes}
\usepackage{enumitem}

\newtheorem{theorem}{Theorem}[section]
\newtheorem{lem}[theorem]{Lemma}
\newtheorem{assumption}{Assumption}[section]
\newtheorem{defn}{Definition}[section]
\newtheorem{prop}{Proposition}[section]
 
 \newcommand{\intR}{\int_{\mathbb{R}^2}}
 \newcommand{\dr}{\mathrm{d}\rho}
 \newcommand{\dR}{\mathrm{d}R}
 
 \DeclareMathOperator{\sgn}{sign}
 \definecolor{darkblue}{rgb}{0,0,0.7}
 \definecolor{lightblue}{rgb}{0,0.4,0.7}

\title{Steady states of an Elo-type rating model\\ for players of varying strength}
\author[1]{Bertram D\"uring}
\author[1]{Josephine Evans}
\author[1]{Marie-Therese Wolfram}

\affil[1]{Mathematics Institute, University of Warwick, Coventry CV4 7AL, United Kingdom}
\date{}

\begin{document}

\maketitle

\begin{abstract}
    \noindent In this paper we study the long-time behaviour of a kinetic formulation of an Elo-type rating model for a large number of interacting players with variable strength. The model results in a non-linear mean-field Fokker-Planck equation and we show the existence of steady states via a Schauder fixed point argument. Our proof relies on the study of a related linear equation using hypocoercivity techniques.
\end{abstract}


\section{Introduction}

In 2015 Jabin and Junca \cite{JJ2015} introduced a kinetic version of the Elo rating model for two player zero sum games. The Elo model was originally introduced by the Hungarian physicist Arpad Elo to rank chess players, but variants of it are nowadays used in bingo, football, basketball and American football. In this model players are characterised by their strength $\rho$, which is an unobservable characteristic, and their rating $R$, which is observable. In the ideal situation the rating $R$ of a player converges to their strength $\rho$ over time as ratings are updated.

We start by recalling the kinetic version of the Elo rating model \cite{JJ2015}. Consider $N\in\mathbb{N}$ players who participate in a sequence of two-player zero sum games. Each player is characterised by their  respective rating $R_\ell$ (which is observable) and strength $\rho_\ell$ (which is unobservable), $\ell \in \lbrace 1,\dots, N \rbrace$. When two players $i,j\in\{1,\dots,N\}$ play a game, their ratings are updated after the encounter using the following binary interaction rule:
\begin{subequations}
\label{eq:orig_elo_inter}
\begin{align}
 R_i^* &= R_i + K\left(S_{ij} -b(R_i - R_j) \right),\\
 R_j^* &= R_j + K\left(-S_{ij} - b(R_j - R_i)\right),
\end{align}
\end{subequations}
where $K$ is a positive constant and $S_{ij}$ the outcome of the game. The random variable $S_{ij}$ takes values $\pm 1$ and, in average, outcomes are assumed to depend on the difference in the underlying strength, that is $\mathbb{E}(S_{ij}) = b(\rho_i - \rho_j)$. Note that this assumption can be generalised to include, for example, draws, then $S_{ij} \in \lbrace -1, 0, 1\rbrace$ or to consider continuous random variables $S_{ij}$ on the interval $[-1,1]$. The function $b$ is usually set to
\begin{align*}
    b(z) = \textrm{tanh} (c z) \quad \text{ with } c \in \mathbb{R}.
\end{align*}
In general $b$ is assumed to be an odd, sufficiently smooth function. Hence, in \eqref{eq:orig_elo_inter} the expected outcome of the game based on the difference in ratings, that is $b(R_i - R_j)$, is compared to the actual score and the ratings are adjusted accordingly. \\
\noindent Jabin and Junca showed that the distribution of players $f = f(\rho, R, t)$ satisfies the following Fokker-Planck equation in the quasi-invariant limit in \cite{JJ2015}:
\begin{align}\label{eq:fpe}
    \partial_t f(\rho, R, t) = \partial_R \left(a[f] f(\rho,R,t)\right)\quad \text{ in }\Omega\times(0,\infty),
\end{align}
with a given initial distribution $f(\rho, R, t) =f_0(\rho, R)$, and $\Omega \subseteq \mathbb{R}^2$. The operator $a$ is given by
\begin{align}
\label{eq:a}
    a[f] = \int_{\Omega} \left(b(\rho-\rho') -b(R-R')\right) f(r',\rho',t) \, \dr' \dR'.
\end{align}
We observe that the distribution of players with respect to their ratings is translated by the `velocity' $a$, which depends on the difference between the expected outcome and the actual score (integrated against the agent distribution).\\ 
Note that interactions \eqref{eq:orig_elo_inter} are translation-invariant on $\mathbb{R}$, since they depend on the difference between ratings only. Similarly, if the initial datum $f_0$ of the Fokker-Planck equation \eqref{eq:fpe} is shifted by constants $\rho_0$ and $R_0$ in $\mathbb{R}^+$, that is $g_0 = f_0(\rho + \rho_0, R+R_0)$, then $g(\rho, R, t) = f(\rho + \rho_0, R + R_0,t)$ is the solution to \eqref{eq:fpe}. Using energy arguments, Jabin and Junca \cite{JJ2015} show that solutions $f$ to \eqref{eq:fpe} concentrate on the diagonal as $t\rightarrow \infty$. Hence, the observable ratings $R$ are guaranteed to converge to the unobservable strengths $\rho$, giving justification for the validity of the kinetic Elo rating model in the many-agents and long-time limit.

Originating in statistical mechanics, in particular in rarefied gas dynamics, Boltzmann-type and Fokker-Planck-type equations (and other kinetic models) have found new applications in socio-economic applications in the past two decades, see \cite{ToscaniPareschi:ims} for an overview. Applications aside from Elo-type rating models \cite{JJ2015,DTW2019,duringwolframfischer2021elo} include wealth distribution in societies \cite{cordier2005kinetic,during2007hydrodynamics,During:economykineticmodel,during2018kinetic}, opinion formation \cite{Toscani:opinion,During:strongleaders,albi2017opinion,Degond2017,During:inhomogeneous,during2021kinetic}, compartmental epidemiology \cite{dimarco2021kinetic,albi2021control} and others.

A generalisation of the original Elo model \eqref{eq:fpe} with variable underlying strength $\rho$ was proposed and investigated by D\"uring et al.\ \cite{DTW2019}. Here the players' strength changes in encounters and is influenced by random fluctuations (introducing an additional diffusion term). We outline the details of the modelling in Section~\ref{s:modelling}, but note that the corresponding player distribution satisfies the following nonlinear Fokker-Planck equation,
\begin{align}
\label{eq:genfpe}
     \partial_t f(\rho, R, t) = \partial_R \left(a[f] f(\rho,R,t)\right) + \gamma \partial_{\rho} \left(a_1[f] f(\rho,R,t)\right) + \frac{\sigma^2}{2} \partial_{\rho}^2 f(\rho,R,t)\quad \text{ in }\Omega\times(0,\infty),
\end{align}
where $\gamma > 0$ and 
\begin{align*}
    a_1[f] = \int_{\Omega} b(\rho-\rho') f(\rho', r',t)\, \dr' \dR'.
\end{align*}
In \cite{DTW2019} existence of weak solutions to \eqref{eq:genfpe} was proved (for $\Omega \subset \mathbb{R}^2$) and numerical experiments illustrated the dynamics of the model. Some heuristic arguments on the long-time behaviour of solutions to \eqref{eq:genfpe} were given, but no rigorous analysis carried out.  In this paper we present an existence results of steady states to equation \eqref{eq:genfpe}. 
Since the diffusion part in \eqref{eq:genfpe} is singular,
the equation is degenerate parabolic. Degenerate Fokker-Planck
equations frequently, despite their lack of coercivity, exhibit
exponential convergence to equilibrium, a behaviour which has been
referred to by Villani as {\em hypocoercivity\/} \cite{Villani2009}. This was subsequently extended \cite{DMS} to a wide variety of kinetic equations.

For linear kinetic Fokker-Planck equations most existing quantitative equilibration results are confined to equations with explicit steady states, or even with linear (in the variables) drift terms, where the whole semigroup can be written explicitly. It is sometimes possible to work with non-explicit steady states as in \cite{CalvezRaoulSchmeiser}. This result requires precise bounds on the non-explicit steady state. The other option is to work in a perturbative setting around an equation where the steady state is know as in \cite{BouinHoffmannMouhot}. In the present situation of \eqref{eq:genfpe}, however, there is no obvious choice of equation to perturb around. The non-explicitness of the steady state encountered here is similar to the theory of non-equilibrium steady states in thermodynamics which appears in the kinetic theory of gasses, see the review article \cite{non-equilibrium} and references therein.

Linking the linear theory of hypocoercivity to non-linear equations is usually done through a linear stability analysis. This provides another motivation for finding the steady state as it gives us something to linearise around. The Vlasov Poisson Fokker-Planck equation is structurally similar to equation \eqref{eq:genfpe} studied here. The linearised problem for this equation is studied in \cite{VlasovPoissonFokker-Planck}. Often for non-linear kinetic equations the challenge in this is to link the (typically small) spaces in which we can study the linearised equations with (typically larger) spaces in which we expect the fully non-linear equation to be well-posed. This was done for the Boltzmann equation in \cite{MischlerMouhotGualdani}. As our equation is well-posed in spaces with exponential weights (due to the boundedness of the non-linear drift terms) we do not expect the same kind of problems to appear here. Studying the long-time behaviour of the fully non-linear equation \eqref{eq:genfpe} is in general very challenging. This problem is strongly linked to the problem of uniqueness for a steady state. We would need to have an entropy function which works for data far from the steady state. For the Boltzmann equation and the kinetic Fokker-Planck equation this role is played by the Boltzmann entropy. There are very few fully non-linear results giving conditional convergence to equilibrium results  \cite{DVFokkerPlanck,DVBoltz}. 

This paper is organised as follows: we discuss the underlying modeling assumption of the generalised Elo model proposed by D\"uring et al.\ \cite{DTW2019} in Section~\ref{s:modelling} and illustrate the dynamics of solutions with computational experiments in Section~\ref{s:numerics}. Section~\ref{s:existence_stat_sol} presents the main contribution of the paper -- the existence of steady solutions. We conclude by discussing the link to hypocoercivity and future research directions in Section~\ref{s:outlook}.

\section{The kinetic Elo rating system with learning effects}

In this section we recall the underlying modeling assumptions and present computational results, that motivate and guide the presented analytical results.

\subsection{Modelling}\label{s:modelling}
We start by discussing the generalisation of the kinetic model proposed by D\"uring et al.\  \cite{DTW2019}. We recall that the rating $R$ should ideally correspond to the underlying strength $\rho$, giving a way to rank and compare players, whose intrinsic strength is not observable.

While the original Elo model assumes that the underlying players' strength is constant, D\"uring et al.\ proposed that players' strength changes over time in various ways: (i) gain of strength by learning from encounters, which depends on the strength difference between the players; (ii) gain or loss of self-confidence due to winning or being defeated in a game; and (iii)
random performance fluctuations. In particular, they generalised the interaction rule \eqref{eq:orig_elo_inter} to
\begin{subequations}
\label{eq:elo_inter}
\begin{align}
 R_i^* &= R_i + K\left(S_{ij} -b(R_i - R_j) \right),\\
 R_j^* &= R_j + K\left(-S_{ij} - b(R_j - R_i)\right),\\
 \rho_i^* &= \rho_i + \gamma h(\rho_j - \rho_i) +  \eta,\\
 \rho_j^* &= \rho_j + \gamma h(\rho_i - \rho_j) + \tilde{\eta},
\end{align}
\end{subequations}
where $\eta$ and $\tilde{\eta}$ are independent and identically distributed random variables with mean zero and variance $\sigma^2$ (which account for daily fluctuations in the individual performance). The function $h$ models individual learning. D\"uring et al.\ proposed the following form (with parameters $\alpha,\beta\ge 0$)
\begin{align*}
    h(\rho_i - \rho_j) = \alpha h_1(\rho_i - \rho_j) + \beta h_2(\rho_i- \rho_j, S_{ij}),
    \end{align*}
where $h_1$ corresponds to the increase of strength and knowledge from encounters,
\begin{align}\label{eq:h1}
    h_1 (z) = 1+ b(z).
\end{align}
Since $h_1$ is positive, both players are able to learn and improve in each game, depending on the difference in strengths, and with a player with lower strength benefiting more. The function $h_2$ models the gain or loss of self-confidence if a player wins a game or is defeated in it -- it is either positive or negative depending on the actual outcome of the game and the expected outcome based on the difference in the players' strength. Then the distribution of players $f$ satisfies the following Fokker-Planck equation in the quasi-invariant limit:
\begin{align*}
     \partial_t f(\rho, R, t) = \partial_R \left(a[f] f(\rho,R,t)\right) + \partial_{\rho} \left(c[f] f(\rho,R,t)\right) + \frac{\sigma^2}{2} \partial_{\rho}^2 f(\rho,R,t)\quad \text{ in }\Omega\times(0,\infty),
\end{align*}
where we assumed that the initial data $f_0$ is normalised, the integral operator $a$ is given by \eqref{eq:a} and 
\begin{align}
    \label{eq:c}
    c[f] = \int_{\Omega} \left[\alpha h_1(\rho'-\rho) + \beta h_2(\rho'-\rho)\right]f(\rho', R', t) \, \mathrm{d}\rho' \dR'.
\end{align}
\noindent Since the positive function $h_1$ results in a continuous increase of the players' strength $\rho$ over time, D\"uring et al.\ studied a suitably shifted problem, which has a steady state. To this end they define
\begin{align*}
    H(\rho, R, t) = \alpha t,
\end{align*}
and the function 
\begin{align}
    g(\rho, R, t) :=  f(\rho + H(\rho, R, t), R, t).
\end{align}
Then the function $g$ satisfies the following Fokker-Planck equation
\begin{align}
\label{eq:g}
     \partial_t g(\rho, R, t) = \partial_R \left(a[g] g(\rho,R,t)\right) + \partial_{\rho} \left(\tilde{c}[g] g(\rho,R,t)\right) + \frac{\sigma^2}{2} \partial_{\rho}^2 g(\rho,R,t)\quad \text{ in }\Omega\times(0,\infty),
\end{align}
where 
\begin{align*}
    \tilde{c}[g] = \int_{\Omega} \left[ \alpha b(\rho'-\rho) + \beta h_2(\rho'-\rho)\right] g(\rho', R',t)\, \dr' \dR'.
\end{align*}
Equation \eqref{eq:genfpe} corresponds to \eqref{eq:g} with $\beta =0$, i.e.\ players improve their strength by participating in games, but do not loose or gain confidence due to wins and losses.
The presented analysis investigates the steady states of this suitably shifted variant of the Elo model.

\subsection{Numerical simulations}\label{s:numerics}

In the following we present several computational experiments which motivate and corroborate our analysis.

 We simulate \eqref{eq:genfpe} on the unit square with no-flux boundary conditions, using Strang splitting and an upwind finite volume discretisation in $\rho$ and $R$ proposed in \cite{DTW2019}.  The domain is discretised into squares of length $h = \frac{1}{800}$, the discrete time steps are set to $\Delta t = 2 \times 10^{-6}$. We assume that agents are initially uniformly distributed, hence $f(\rho, R, t=0) \equiv 1$.
 
 Figure \ref{fig:Ufinal} shows the computationally obtained steady state of players  $f_{\infty}$ for $\frac{\sigma^2}{2} = 0.05$  after $2\times 10^5$ time steps. We observe the formation of a smoothed peak  at $(\rho,R) = (0.5, 0.5)$ (the centre of mass). The smaller the diffusivity, the more concentrated $f_{\infty}$ (converging to the expected Delta Dirac steady state in the case $\sigma^2 = 0$).
\begin{figure}
    \centering
    \includegraphics[width=0.6\textwidth]{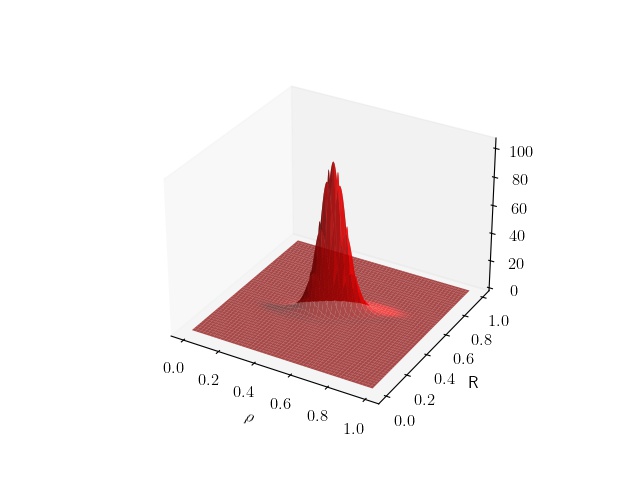}
    \caption{Steady state player distribution $f_{\infty}$.}
    \label{fig:Ufinal}
\end{figure}
Figure \ref{fig:rel_entropies} illustrates the evolution of different weighted relative energies, which we will investigate in Section \ref{s:existence_stat_sol}. In particular, we consider 
\begin{align}\label{e:energy}
    E(f; f_{\infty})=\intR \phi(\rho,R) \lvert f(\rho, R, t) - f_{\infty}(\rho,R)\rvert \, \dr \dR
\end{align}
with a weight function $\phi=\phi_\beta = \exp \left( \beta\sqrt{1+4\rho^2/\gamma + 2\rho R + \gamma R^2}\right)$
or $\phi = f^{-1}_{\infty}$. We note that $\phi_{\beta}$ will be used in the analysis later.
We compute the relative energies with respect to the computationally obtained steady state density $f_{\infty}$ (we see in Figure \ref{fig:rel_entropies} that the solution has indeed equilibrated).
\begin{figure}[ht!]
    \centering
\subfigure[Weighted relative energies for $\phi = \phi(\beta,\gamma)$.]{\includegraphics[width=0.4\textwidth]{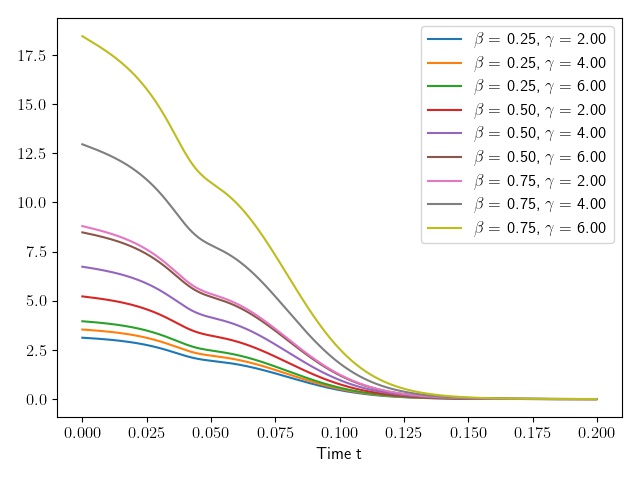}}
\hspace*{1em}
\subfigure[Weighted relative energy for $\phi = f^{-1}_{\infty}$]{\includegraphics[width=0.425\textwidth]{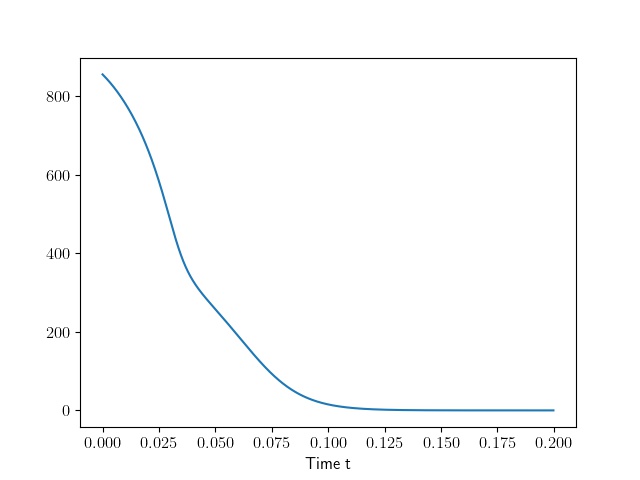}}
    \caption{Decay of the relative energy \eqref{e:energy} for different weights functions $\phi = \phi(\beta, \gamma)$ defined by \eqref{eq:definephi} and for $\phi = f_{\infty}^{-1}$.}
    \label{fig:rel_entropies}
\end{figure}
The evolution of \eqref{e:energy} shows a staircase like decay -- a phenomena often observed in linear kinetic Fokker Planck equations \cite{arnoldrevista} and the Boltzmann equation \cite{FilbetMouhotPareschi}.


\section{Existence of steady states}
\label{s:existence_stat_sol}

\subsection{Main Results}

The previous discussion motivates the main result of this paper -- an existence proof of steady states of the following equation:
\begin{equation} \label{eq:nlmain}
    \partial_t f = -\partial_R \left(a[f]f\right) + \partial_\rho \left( \frac{\sigma^2}{2} \partial_\rho f + \gamma a_1[f]f \right)\quad \text{ in }\mathbb{R}^2\times(0,\infty),
\end{equation}
with $a[\mu](\rho, R) = a_1[\mu](\rho) - a_2[\mu](R)$, where
\begin{align*}
    a_1[\mu](\rho) &= \intR b(\rho-\rho')\mu(\rho',R')\,\mathrm{d}\rho' \mathrm{d}R',\\
    a_2[\mu](R) &= \intR b(R-R') \mu(\rho', R')\,\mathrm{d}\rho' \mathrm{d}R'.
\end{align*}  
Equation \eqref{eq:nlmain} is supplemented with initial condition $f(\rho,R,0) = f_0(\rho,R)$ with $\intR f_0\, \dr \dR = 1$. Note that \eqref{eq:nlmain} corresponds to \eqref{eq:g} with $\beta=0$ and $\alpha = \gamma$.

We make the following assumption on the function $b$.
\begin{assumption} \label{assumptionb}
We assume that $b$ is an odd, smooth function on $\mathbb{R}$ and that there exist constants $\alpha>0$ and $C>0$ such that
\[|1-b(|z|)| \leq Ce^{-\alpha \langle z \rangle}.\]
\end{assumption}
We also recall the function $\phi_\beta$ which is used in the analysis,
\begin{equation} \label{eq:definephi}
\phi_\beta = \exp \left( \beta\sqrt{1+4\rho^2/\gamma + 2\rho R + \gamma R^2}\right). 
\end{equation}
The following theorem states the main result of the paper.
\begin{theorem}\label{thm:main}
Let Assumption \ref{assumptionb} be satisfied. Then there exists a steady solution $f_\infty$ to \eqref{eq:nlmain}, which is a probability measure on $\mathbb{R}^2$. Furthermore, $f_\infty$ is a smooth function with a bounded exponential moment.
\end{theorem}
We will prove existence of a steady state using Schauder's fixed point argument. In doing so we consider the following linearised transport equation 
for fixed probability measure $\mu$, 
 \begin{equation} \label{eq:linearmain}
 \partial_t f = -\partial_R \left( a[\mu] f \right) + \partial_\rho \left( \frac{\sigma^2}{2} \partial_\rho f +  \gamma a_1[\mu](\rho) f \right).
 \end{equation} 
     \noindent We will prove the existence of a steady state using the following steps:
 \begin{itemize}
     \item Step 1: Define a map $G(\mu)$, from the set of probability measures on $\mathbb{R}^2$ to itself by setting $G(\mu)$ to be the unique steady state of \eqref{eq:linearmain}.
     \item Step 2: Show that $G(\mu)$ is well-defined by a Harris's theorem argument. This step will also give us some bounds on $G(\mu)$. We will define a scale of exponential moments  of the form $\intR \mu \phi_\beta \dr \dR$ and get bounds on $\intR G(\mu)\phi_\beta \, \dr \dR$ for $\beta$ sufficiently small.
     \item Step 3: We show that under Assumption \ref{assumptionb} our bounds on $G(\mu)$ will allow us to show that if 
     \[ \intR \mu \phi_\beta \, \dr \dR \leq M, \] then
     \[ \intR G(\mu) \phi_\beta \, \dr \dR \leq CM^{\eta}, \] where $\eta$ is a constant we can find explicitly in terms of $\alpha$ from Assumption \ref{assumptionb} and $\beta$. When $\beta$ is small enough we will have $\eta < 1$ which allows us to find a convex set of measures which is preserved by the function $G$. 
     \item Step 4: We then need to choose a topology to show that $G$ is a continuous map with a compact image. The results of step 2 will provide us with a natural topology and allow us to relate continuity of $G$ with respect to $\mu$ to continuity of the semigroup associated to the linear equation \eqref{eq:linearmain} with respect to $\mu$. We show continuity of the semigroup in Wasserstein distance and then convert this to continuity in our strong topology using an interpolation argument and the regularising nature of \eqref{eq:linearmain}. We have weak compactness as a result of our moment estimates.
     \item Step 5: The above allows us to apply Schauder's fixed point theorem to the map $G$. This gives us the existence of a steady state to the non-linear equation. Furthermore, from our study of $G$ and the fact that the steady state is a fixed point of $G$ this also give us that the steady state has a smooth density (from the regularising estimates in step 4) and $\intR f\phi \, \dr \dR \leq M$ (from the bounds found in step 3). 
 \end{itemize}
 
 \subsubsection{Discussion of the linear equation \eqref{eq:linearmain}}
 Before we begin the main technical proofs we comment that the analysis relies strongly on an understanding of both the steady states and convergence to the steady states for \eqref{eq:linearmain}. This is a linear Fokker-Planck equation and is in some ways very similar to equations which have been studied using hypocoercivity theory, particularly by Villani in \cite{Villani2009}.
 
 A key example of an equation to which we can apply the theory in \cite{Villani2009} is the kinetic Fokker-Planck equation of the following form:
 \begin{equation} \label{eq:kfp}
     \partial_t f = - v \partial_x f + \partial_x \phi(x) \partial_v f + \partial_v \left( \frac{\sigma^2}{2} \partial_v f + v f\right).
 \end{equation}
 Here the spatial variable $x$ corresponds to the ratings $R$ and the velocity $v$ to the underlying strength $\rho$.
 We can make a comparison of the different terms in this equation with terms in \eqref{eq:linearmain}. 
 \begin{itemize}
    \item $\sigma^2 \partial^2_v f$ and $\sigma^2 \partial_\rho^2 f$ both correspond to linear second order diffusions which are only present in one of the two variables.
     \item The transport terms in $x$ and $R$, that is $-v \partial_x$ and $-\partial_R(a_1[\mu](\rho) f)$, correspond to mixing between the diffused and undiffused variables.
     \item The transport terms in $v$ and $\rho$, that is $\partial_v(vf)$ and $\gamma\partial_\rho(a_1[\mu](\rho) f)$, correspond to confining effects in the diffused variables.
     \item The terms $\partial_x \phi(x) \partial_v f$ and $\partial_R(a_2[\mu](R) f)$ both correspond to confining terms in the undiffused variables. We should note here that these two terms are less directly comparable.
 \end{itemize}
 
 For both equations we expect the term that mixes between the diffused and undiffused variables to allow the smoothing and spreading effect from the diffusion operator to affect the non-diffused variables. This is the key effect in both hypoellipticity and hypocoercivity. We expect the combination of this with the presence of confining terms in both variables to allow us to show convergence to equilibrium.
 
 However, there are differences between \eqref{eq:linearmain} and \eqref{eq:kfp} which mean the study of their long time behaviour is very different. The most important of these differences is is that for \eqref{eq:kfp} we can write down an explicit steady state $\exp( - \phi(x) - v^2/2\sigma).$ This is not the case for \eqref{eq:linearmain}. This does not just affect the amount of information we have about the steady state it also vastly reduces the tools we have for studying the convergence to equilibrium. Almost all hypocoercivity theories including \cite{Villani2009} and \cite{DMS} require us to work in spaces weighted against the equilibrium state and to use explicit knowledge of this state when working in these spaces. There are a few works where they are able to apply these theorems to equations with non-explicit steady states in a perturbative setting \cite{BouinHoffmannMouhot, Iacobucci}. The only work we are aware of in a non-perturbative setting is \cite{CalvezRaoulSchmeiser}. We note that Harris's theorem in a good tool in this situation and has been successfully applied to linear kinetic equations with non-explicit steady states in \cite{Cao, Bernou, EvansYoldas, EvansMenegaki}.
 
 Another additional challenge when working with \eqref{eq:linearmain} as compared to \eqref{eq:kfp} is that the term that mixes the diffused and undiffused variables is weaker in \eqref{eq:linearmain}. Specifically, in the kinetic Fokker-Planck equation we will see a diffusive effect similar to what would be produced by $[v\partial_x, \partial_v]^2 = \partial^2_x$ at times of order $t^2$. Working analogously for \eqref{eq:linearmain} we expect to generate a diffusion in the $R$ variables at second order in time similar to what would be produced by the operator $[a_1[\mu](\rho) \partial_R, \partial_\rho]^2 = (\partial_\rho a_1[\mu](\rho))^2 \partial^2_R$. But since $\partial_\rho a_1[\mu](\rho) \rightarrow 0$ exponentially fast as $|\rho| \rightarrow \infty$ we see that the diffusion effect in the $R$ variables becomes weaker and weaker as $|\rho| \rightarrow \infty.$ Similar effects were observed in case of a special relativistic kinetic Fokker-Planck equation, see \cite{Calogero}. We expect this weak mixing to be especially challenging in our setting as our confining terms are bounded. This means we expect to have exponential rather than Gaussian concentration of the steady states. The combination of these two effects means that even if we had an explicit form for the steady states the Poincar\'e inequalities needed in the theory of \cite{Villani2009} would not be valid. In this respect we are also helped by using Harris's theorem. As we will see in more detail in the proof verifying the assumptions of Harris's theorem only requires us to verify the mixing property of the semigroup on a compact set. This means we do not need to worry that our mixing effect becomes 0 as $|\rho| \rightarrow \infty.$

\subsubsection{Topologies and spaces}
We will use the following function spaces, topologies and theorems in the main proof. We work in the space $\mathcal{P}_\beta$ defined by
\begin{defn}
We define $\mathcal{P}_\beta$ to be the space of probability measures on $\mathbb{R}^2$ which have the property that
\[ \int_{\mathbb{R}^2} \phi_\beta(\rho, R)\mu(\mathrm{d}\rho, \mathrm{d}R) < \infty. \]
Here $\phi_\beta$ is as defined in \eqref{eq:definephi}.
\end{defn}
We consider this as a subspace of the space $\mathcal{M}_\beta$ defined by
\begin{defn}
We define $\mathcal{M}_\beta$ to be the space of signed measures on $\mathbb{R}^2$ which satisfy
\[ \int_{\mathbb{R}^2} \phi_\beta(\rho, R)|\mu|(\mathrm{d}\rho, \mathrm{d}R) < \infty.\]
\end{defn}
We can make $\mathcal{M}_\beta$ a Banach space with the following norm.
\begin{defn}
We define a norm $\|\cdot\|_\beta$ on $\mathcal{M}_\beta$ by
\[ \|\mu\|_\beta = \int_{\mathbb{R}^2} \phi_\beta(\rho, R)|\mu|(\mathrm{d}\rho, \mathrm{d}R) < \infty.  \]
\end{defn}
We consider two topologies on $\mathcal{P}_\beta$: the topology defined by the norm $\|\cdot\|_\beta$ and the topology defined by the notion of weak convergence of measures. We recall Prokhorov's theorem.
\begin{theorem}[Prokhorov's Theorem]
Suppose that $S$ is a metric space and $\mathcal{P}(S)$ is the space of probability measures on $S$, then a collection of probability measures $\mathscr{C}$ is sequentially compact in the topology of weak convergence of measures if and only if it is tight. Where we say the set $\mathscr{C}$ is tight if for every $\epsilon>0$ there exists a compact set $K_\epsilon \subset S$ such that for any $\mu \in \mathscr{C}$ we have
\[ \mu(S \setminus K_\epsilon) < \epsilon.\]
\end{theorem}
\begin{lem} \label{lem:compact}
Any subset $\mathscr{C}$ of $\mathcal{P}_\beta$ which is bounded with respect to the $\|\cdot\|_*$ norm is sequentially compact in the topology of weak convergence of measures.
\end{lem}
\begin{proof}
By Prokhorov's theorem it is sufficient to prove that such a set is tight. Suppose that for every $\mu \in \mathscr{C}$ we have
\[ \int_{\mathbb{R}^2} \phi(\rho, R) \, \mu(\mathrm{d}\rho, \mathrm{d}r) \leq M. \] Then we know that $\phi^{-1}$ converges to $0$ as $\|(\rho, R)\| \rightarrow \infty$ so for every $\epsilon$ there exists a compact set $K_\epsilon$ such that
\[ \phi^{-1}(\rho,R) < \epsilon/M, \; \forall (\rho, R) \notin K_\epsilon. \] Using this we have
\begin{align*} 
\int_{\mathbb{R}^2 \setminus K_\epsilon} \mu(\mathrm{d}\rho, \mathrm{d}R) & = \int_{\mathbb{R}^2} \left( \phi(\rho, R)^{-1} 1_{\mathbb{R}^2 \setminus K_\epsilon} \right) \left( \phi(\rho, R) \right)\mu(\mathrm{d}\rho, \mathrm{d}R)\\
& \leq \frac{\epsilon}{M} \intR \phi(\rho,R)\mu(\mathrm{d}\rho, \mathrm{d}R)\\ 
&\leq \epsilon.
\end{align*}
\end{proof}
Finally, we define the compact set used in Schauder's fixed point theorem.
\begin{defn} \label{dfn:C}
We write $\mathscr{C}_{M,\beta}$ to be the set 
\[ \mathscr{C}_{M,\beta} := \{ \mu \, | \, \mu \in \mathcal{P}_{\beta}, \|\mu\|_* \leq M\},\] for positive constants $\beta, M$. We note that by the results above $\mathscr{C}_{M,\beta}$ is compact in the topology of weak convergence of measures for any $M, \beta$.
\end{defn}

\subsection{Schauder fixed point argument}
In order to apply Schauder's fixed point argument we first need to define a function which will have a steady state of $f_\infty$ as a fixed point. 
\begin{defn}
    For fixed $\mu$, a probability measure, we write $\mathcal{S}_t^\mu$ to be the linear semigroup associated to the evolution governed by equation \eqref{eq:linearmain}. Furthermore, when this equation has a unique steady state we write $G(\mu)$ to be this steady state.
\end{defn}
\subsubsection{The function $G$ is well-defined}
In this section we prove the following proposition.
\begin{prop} \label{prop:Gdefn}
For any probability measure $\mu$ on $\mathbb{R}^2$ there exists a unique steady state to equation \eqref{eq:linearmain}. Therefore, the function $G$ is well-defined.
\end{prop}
We prove this proposition after showing the following lemma first.
\begin{lem} \label{lem:rmu}
For any $\mu$ there exists $z_1$ such that $|a_1[\mu](\rho)|> 1/2,$ whenever $|\rho|>z_1$, and $z_2$ such that $|a_2[\mu](R)|>1/2$, whenever $|R| >z_2$.
\end{lem}
\begin{proof}
We have that $b(\rho - \rho') \rightarrow c$ as $\rho \rightarrow \infty$. So by dominated convergence $a_1[\mu](\rho) \rightarrow c$ as $\rho \rightarrow \infty$ and $a_1[\mu](\rho) \rightarrow - c$ as $\rho \rightarrow \infty$. The same behavior holds true for the operator $a_2[\mu](R)$, but this time as a function of $R$. We can also differentiate these functions to see that they are monotonically increasing. This gives the result.
\end{proof}
As a consequence of the proof techniques in the proposition we also find the following result.
\begin{lem} \label{lem:Gbound}
For $\beta$ sufficiently small there exists some positive constants $C, \lambda$ such that for any two initial data $f_1, f_2 \in \mathcal{P}_{\beta}$ we have
\[ \| S_t^\mu(f_1 - f_2)\|_\beta \leq Ce^{-\lambda t} \|f_1 - f_2\|_\beta  \]
Then there exits a constant $D>0$, not depending on $\mu$, such that
\[ \int_{\mathbb{R}^2} G(\mu) \phi \, \mathrm{d}\rho \mathrm{d}R \leq D \sup_{|(\rho',R')| \leq r(\mu)} \phi(\rho',R').  \] 
Here, the function $r(\mu)$ is defined such that $|(\rho',R')|<r(\mu)$ implies that $|\rho|<z_1$ and $|R| \leq \max\{ 1,2\gamma\}z_2$ and $|(\rho, R)| \leq z_3$.
\end{lem}

In this section we work with stochastic tools. Before we begin doing this we relate the linear equation \eqref{eq:linearmain} to a Markov process. In particular, equation \eqref{eq:linearmain} is the Kolmogorov forward equation (equation which evolves the law forward in time) related to the SDE for the continuous in time stochastic processes $R_t$ and $\rho_t$, 
\begin{subequations}
\label{eq:SDEmain}
\begin{align} 
    \mathrm{d}R_t &= a[\mu](\rho_t, R_t)\,\mathrm{d}t,\\
    \mathrm{d}\rho_t &= -a_1[\mu](\rho_t)\,\mathrm{d}t + \sigma\, \mathrm{d}B_t,
\end{align} 
\end{subequations}
where $B$ is a Brownian motion.

We prove both Proposition \ref{prop:Gdefn} and Lemma \ref{lem:Gbound} by applying Harris's theorem. We use the version of Harris's theorem found in \cite{Mattingly}, which we restate here. First we need to state their assumptions. The theorem is for a Markov chain $X_t$ with transition kernel $P_t(x,A)$. Then the semigroup associated to our PDE will be given by $(S_t^\mu f_0)(A) = \int_A \int P_t(x,y)f_0(x)\, \mathrm{d}x \mathrm{d}y = \int_A f_t(y) \,\mathrm{d}y.$ We also have our PDE \eqref{eq:linearmain} is written $\partial_t f = \mathcal{L}f,$ where $\mathcal{L}$ is the generator of $S_t^\mu.$ Then $\mathcal{L}^*$ is the formal adjoint of $\mathcal{L},$ and the generator of the semigroup $(S^{\mu *}_t \phi)(x) = \int \mathcal{P}_t(x,\mathrm{d}y)\phi(y).$
\begin{assumption} \label{a:harris1}
The transition kernel has to satisfy the following two assumptions:
\begin{itemize}
    \item For any compact set $K$, there is some $y_*$, such that for any $\delta>0$ there exists $t(\delta)$ such that $P_{t(\delta)}(x,B(y_*,\delta))>0 \quad \forall x \in K.$
    \item For every $t$ the transition kernel possesses a density $p_t(x,y)$ which is jointly continuous in $x,y$ everywhere.
\end{itemize}
\end{assumption} 
The second assumption is as follows.
\begin{assumption} \label{a:harris2}
For some fixed $T$
    there exists a non-negative function $V$ with $V(x) \rightarrow \infty$ as $|x| \rightarrow \infty$ and two constants $\alpha \in (0,1)$ and $C \geq 0$ such that 
    \[ \mathbb{E}(V(X_{(n+1)T}) \,|\, X_{nT}) \leq \alpha V(X_{nT}) + C. \]
    It is standard to verify this assumption by proving that
    \[ \mathcal{L}^*V \leq - \lambda V + C, \] where $\mathcal{L}^*$ is the formal adjoint of the generator of the semigroup and $\lambda, C >0.$
\end{assumption}
We then have the Theorem
\begin{theorem}[Harris's theorem, as in \cite{Mattingly}]
If we have a Markov process $X_t$ that satisfies both Assumption~\ref{a:harris1} and Assumption~\ref{a:harris2}. Then there exists a unique steady state probability measure for the Markov semigroup $P_t$. Furthermore, we have that there exists constants $C>0$ and $\lambda>0$ such that
\[ \|P_t (f_1 - f_2)\|_{V} \leq C e^{-\lambda t}\|f_1 - f_2\|_{V}.  \] where
$\|f\|_V = \int V(x)|f|(\mathrm{d}x).$
\end{theorem}

Before verifying the above assumptions we note that a more standard way of writing Harris's theorem is to replace Assumption \ref{a:harris1} by the assumption that the semigroup $S_t^{\mu}$ has a uniformly over $x_0$ in any compact set, lower bound of the form
\[ \mathcal{S}_t^\mu \delta_{x_0} \geq \alpha \nu,  \] where $\alpha \in (0,1)$ and $\nu$ is a probability measure. For our equation an assumption of this form could be verified by applying the result of \cite{BallyKohatsuHiga}. This would also be a more quantitative result. However, precisely applying the Theorem of Bally and Kohatsu-Higa would take a lot of time relative to the less quantitative results given here.\\

We proceed by showing that the semigroup $\mathcal{S}_t^\mu$ satisfies Assumption \ref{a:harris1} and \ref{a:harris2} for any $\mu$.
\begin{lem} \label{lem:doeblin}
The linear semigroup $S_t^\mu$ satisfies Assumption \ref{a:harris1} for any $\mu$.
\end{lem}

Before proving this lemma, we state two useful theorems from stochastic calculus. 
\begin{theorem}[Malliavin's Theorem, see \cite{Nualart}, Section 7.5]
Given a $d$-dimensional SDE in Stratonovich form
\[ dX_t = V_0(X_t)\,\mathrm{d}t + \sum_{k=1}^m V_k(X_t) \circ \mathrm{d}B^k_t, \] with $m$ independent Brownian motions. Furthermore, we define the set of vector fields
\[ \mathcal{V}_0 = \{ V_0\}, \quad \mathcal{V}_{n+1} = \mathcal{V}_n \cup \{ [V_k, U]\,:\, k=1,\dots, m,  \, U \in \mathcal{V}_{n} \} . \] 
Then if there exists an $m$ such that $\mathcal{V}_m$ spans $\mathbb{R}^2$ at each $x$, then the stochastic process $X_t$ has a jointly continuous transition kernel.
\end{theorem}
\begin{theorem}[Strook-Varadhan support theorem, see \cite{StroockVaradhan}]\label{t:SV}
Given a $d$-dimensional SDE in Stratonovich form
\[ dX_t = V_0(X_t)\,\mathrm{d}t + \sum_{k=1}^m V_k(X_t) \circ \mathrm{d}B^k_t, \quad X_0 =x, \] the support of the law of $X_t$ is the closure of the set of points reached at time $t$ by the ODE
\[ \frac{\mathrm{d}x(t)}{\mathrm{d}t} = V_0(x(t)) + \sum_{k=1}^m V_k(x(t)) \frac{\mathrm{d}h_k(t)}{\mathrm{d}t}, \] when we let the $h_k$ range over the set of continuously differentiable functions.
\end{theorem}
\begin{proof}[Proof of Lemma \ref{lem:doeblin}]
This part of the proof is very similar in spirit to the proof of the analagous result for the Langevin equation found in \cite{Mattingly}. The second part of the assumption is an immediate consequence of the hypoellipticity of the equation. We can see that the SDE \ref{eq:SDEmain} satisfies the assumptions to apply Malliavin's version of H\"ormander's theorem.
The vector field $\partial_\rho$ and $[\partial_\rho, a[\mu]\partial_R] = (\partial_\rho a[\mu] (\rho)) \partial_R$ span the tangent space at every point $(\rho, R)$ as $\partial_\rho a[\mu]>0.$ 
 
 For the first part, as in Higham, Stuart and Mattingly \cite{Mattingly}, we use the Strook-Varadhan support theorem \cite{StroockVaradhan}. Given this theorem we fix the point $(\rho_*, R_*)$, which depends on $\mu$, chosen so that $a_1[\mu](\rho_*) = a_2[\mu](R_*) = 0.$ Then fix $\delta$ and a compact set $\mathcal{K}$. We have the control system
 \[ \left( \begin{array}{c} \mathrm{d}R/\mathrm{d}t \\ \mathrm{d}\rho/\mathrm{d}t \end{array}\right) = \left( \begin{array}{c}a_1[\mu](\rho) - a_2[\mu](R) \\ -\gamma a_1[\mu](\rho) + \sigma \mathrm{d}v/\mathrm{d}t \end{array}\right). \] As $a_1[\mu](\rho)$ is a smooth function this has the same reachable sets as the control system
 \begin{equation}
     \left( \begin{array}{c} \mathrm{d}R/\mathrm{d}t \\ \mathrm{d}\rho/\mathrm{d}t \end{array}\right) = \left( \begin{array}{c} a_1[\mu](\rho) - a_2[\mu](R) \\  \mathrm{d}\tilde{v}/\mathrm{d}t \end{array}\right).
 \end{equation}
 Now for any $\rho_0$ we can find $\tilde{v}$ such that $\rho_0 + \tilde{v}(1) = \rho_*$ and $\mathrm{d}\tilde{v}/\mathrm{d}t = 0$ for $t>1$ and $\tilde{v}$ is smooth. With this control we have that $R(1)$ is somewhere in a ball of radius $c$ around $R_0$, so if $R_0$ was in the original compact set $K$, $R(1)$ is now in a new compact set, $K'$. Now there exist a $T = T(\delta, K)$ such that after time $T$ the ODE
 \[ \frac{\mathrm{d}R}{\mathrm{d}t} = -a_2[\mu](R)  \] when started inside $K'$ will be in $B(R_*, \delta)$. Therefore, our control $\tilde{v}$ will move the equivalent control system to a point in $B((\rho_*,R_*),\delta)$ after time $T+1$. Consequently, the Strook-Varadhan support theorem \ref{t:SV} implies that
 \[ P((\rho_t, R_t) \in B(\rho_*, R_*, \delta))>0 \] if $t>T+1.$
\end{proof}

We now move onto the second assumption.
\begin{lem} \label{lem:lyapunov}
The function $\phi_\beta$
is a Foster-Lyapunov function for the semigroup $S_t^\mu$, for $\beta$ sufficiently small. That is to say if $\mathcal{L}$ is the generator associated to $S_t^\mu,$ 
\[ \mathcal{L}^* \phi_\beta \leq -\lambda \phi_\beta + A1_{|(\rho, R)|\leq B}, \] for some strictly positive constants $\lambda, A, B,$ again for $\beta$ sufficiently small. 
\end{lem}
\begin{proof}
First we compute
\begin{align*}
    \frac{\mathcal{L}^*\phi_{\beta}}{\phi_{\beta}} &= \frac{\beta \left( -3a_1[\mu](\rho)\rho - \gamma a_2[\mu](R)R -a_2[\mu](R)\rho \right)}{\sqrt{1+4\rho^2/\gamma + 2\rho R + \gamma R^2}} \\
    &\quad + \frac{\sigma^2}{2} \left(\frac{3\beta R^2}{(1+4\rho^2/\gamma + 2\rho R + \gamma R^2)^{3/2}} + \frac{\beta^2 (R+4\rho/\gamma)^2}{1+4\rho^2/\gamma + 2\rho R + \gamma R^2} \right)
\end{align*}
This implies the following bound from above
\begin{align*}
    \frac{\mathcal{L}^*\phi_{\beta}}{\phi_{\beta}} & \leq \frac{\beta \left( -3a_1[\mu](\rho)\rho - \gamma a_2[\mu](R)R +c|\rho| \right)}{\sqrt{1+4\rho^2/\gamma + 2\rho R + \gamma R^2}}  + \frac{2 \sigma^2 \beta}{\gamma} \left(\frac{1}{\sqrt{1+4\rho^2/\gamma + 2\rho R + \gamma R^2}} + \beta \right).
\end{align*}
Using Lemma \ref{lem:rmu} we have that whenever $|\rho|>z_1$ and $|R|>z_1$ that
\[ -3a_1[\mu](\rho)\rho - \gamma a_2[\mu](R)R +c|\rho|  \leq -c|\rho|/2-c\gamma|R|/2.  \]
Furthermore, if $|\rho|>z_1$ and $|R| < z_2$, then
\[ -3a_1[\mu](\rho)\rho - \gamma a_2[\mu](R)R +c|\rho|  \leq -c|\rho|/2. \]
And if $|\rho|< z_1$ and $|R|> \max\{1, 2/\gamma \}z_2$, we have
\[ -3a_1[\mu](\rho)\rho - \gamma a_2[\mu](R)R +c|\rho|  \leq - \gamma c |R|/4. \] Therefore, there exists some $\Lambda$ which doesn't depend on $\mu$ such that if $|\rho|>z_1$ or $|R|> \max\{1, 2/\gamma\}z_2$, we have
\[ \frac{\mathcal{L}^*\phi_{\beta}}{\phi_{\beta}} \leq - \beta \Lambda +  + \frac{2 \sigma^2 \beta}{\gamma} \left(\frac{1}{\sqrt{1+4\rho^2/\gamma + 2\rho R + \gamma R^2}} + \beta \right). \] Therefore, if $\beta$ is small enough and $|\rho|>z_1, |R|> \max\{ 1, 2/\gamma\} z_2$ and $\rho, R$ large enough so that the second term is small, specifically
\[ |(\rho,R)| \geq z_3:= \frac{\frac{\Lambda\gamma}{4\sigma^2}-\beta}{\sqrt{2/\gamma + \gamma/2}}, \] then
\[ \frac{\mathcal{L}^*\phi_{\beta}}{\phi_{\beta}} \leq - \beta \Lambda /2. \] Therefore, we have
\[  \mathcal{L}^*\phi_{\beta} \leq - \frac{1}{2}\beta \Lambda \phi_{\beta}  +\beta \left(\sqrt{\gamma/3}c + 2\sigma^2 \frac{(1+\beta)}{\gamma} + \frac{\Lambda}{2} \right) \sup_{|\rho| \leq z_1, |R| \leq \max\{1,2\gamma\} z_2, |(\rho, R)|\leq z_3}\phi_{\beta}(\rho', R').  \] These calculations show that $\phi_{\beta}$ satisfies the conditions to be a Foster-Lyapunov function when $\beta$ is sufficiently small.
\end{proof}
 
 We now prove both Proposition \ref{prop:Gdefn} and Lemma \ref{lem:Gbound} simultaneously.
 \begin{proof}[Proof of Proposition \ref{prop:Gdefn} and Lemma \ref{lem:Gbound}]
 Using Lemma \ref{lem:doeblin} and Lemma \ref{lem:lyapunov} we have that the Markov process defined by \eqref{eq:SDEmain} satisfies the condition of Harris's theorem with Lyapunov function $\phi_{\beta}$. This gives the existence and uniqueness of a steady state measure for this SDE, and gives the convergence result in the first part of Lemma \ref{lem:Gbound}. For the other bound in Lemma \ref{lem:Gbound} we recall the final inequality from the proof of Lemma \ref{lem:lyapunov}, \[  \mathcal{L}^*\phi_{\beta} \leq - \frac{1}{2}\beta \Lambda \phi_{\beta}  +\beta \left(\sqrt{\gamma/3}c + 2\sigma^2 \frac{(1+\beta)}{\gamma} + \frac{\Lambda}{2} \right) \sup_{|\rho| \leq z_1, |R| \leq \max\{1,2\gamma\} z_2}\phi_{\beta}(\rho', R').  \] These calculations show that when $\beta$ is sufficiently small, we have, for $G(\mu)$ being the steady state,
\[ \intR G(\mu) \phi_{\beta} \, \dr \dR \leq \frac{2}{\Lambda}\left(\sqrt{\gamma/3}c + 2\sigma^2 \frac{(1+\beta)}{\gamma} + \frac{\Lambda}{2} \right) \sup_{|(\rho', R')|\leq r(\mu)}\phi_{\beta}(\rho', R'). \] This bound only depends on $\mu$ through its explicit dependence on $z_1, z_2$.
 \end{proof}
 
\subsubsection{Finding a set which is preserved by the function $G$}
In this section we prove the following proposition.
\begin{prop} \label{prop:preservedset}
Let $M$ be sufficiently large, $\beta$ sufficiently small and $b$ satisfies Assumption \ref{assumptionb}. Then the sets $\mathcal{C}_{M,\beta}$,  defined in Definition \ref{dfn:C} are preserved by the map $G$.
\end{prop}

We begin with a lemma.
\begin{lem}
Let $\mu \in \mathscr{C}_{M,\beta}$ and $b$ satisfies Assumption \ref{assumptionb}. Then we can give explicit expressions for the constants $z_1$ and $z_2$ appearing in Lemma \ref{lem:rmu} for $M$ large. These are
\begin{align}
z_1 = \log(4MC')/\delta \text{ and } z_2 = \log(4MC')/\delta', 
\end{align}
with constants $\delta$ and $\delta'$ given by
\begin{align*}
    \delta = \frac{2\alpha \beta \sqrt{3}}{\alpha \sqrt{\gamma}+\beta \sqrt{3}}, \text{ and } \delta' = \frac{2\alpha \beta \sqrt{3\gamma}}{\alpha \sqrt{\gamma} + \beta \sqrt{3}}. 
    \end{align*}
    We recall that $\beta$ is a constant which is fixed in the proof of Lemma \ref{lem:lyapunov}.
\end{lem}
\begin{proof}
We begin with
\begin{align*}
    |1-a_1[\mu](\rho)| & = |\intR (1-b(\rho - \rho')) \mu(\rho',R')\, \mathrm{d}\rho'\mathrm{d}R'| \\
    & \leq \intR |1 - b(\rho - \rho')| \mu(\rho', R') \, \mathrm{d}\rho' \mathrm{d}R' \\
    & = \intR |1- b(\rho - \rho')| \frac{1}{\phi(\rho',R')} \phi(\rho',R') \mu(\rho',R') \, \mathrm{d}\rho' \mathrm{d}R' \\
    & \leq \sup_{\rho',R'} \left(|1-b(\rho-\rho')|\frac{1}{\phi(\rho',R')} \right) \intR \phi(\rho',R') \phi(\rho', R')\, \mathrm{d}\rho' \mathrm{d}R' \\
    & \leq C M  \sup_{\rho', R'} \left(e^{-\alpha\langle \rho -\rho'\rangle}\frac{1}{\phi(\rho',R')} \right)\\
    & \leq CM \sup_{\rho'} \left(e^{-\alpha \langle \rho - \rho' \rangle } \exp \left(- \beta \sqrt{1+ \frac{3\rho'^2}{\gamma}} \right) \right) =: M F(\rho).
\end{align*}
Furthermore,
\[ F(\rho) \leq C' \exp \left( - \delta \langle \rho \rangle \right), \] with 
\[ \delta = \frac{ \sqrt{12} \alpha \beta}{\alpha \sqrt{\gamma} + \beta \sqrt{3}}. \] This means that if $\langle \rho \rangle > \log(4MC')/ \delta$, then $|a_1[\mu](\rho)|> 1/2.$\\
In a similar way, we have
\begin{align*}
    \lvert 1-a_2[\mu](R)\rvert & \leq CM \sup_{\rho', R'}\left( e^{-\alpha \langle R-R'\rangle} \frac{1}{\phi(\rho',R')}\right)\\
    & \leq CM \sup_{R'} \left( \exp \left( - \alpha \langle \rho -\rho'\rangle - \beta \sqrt{ 1+ 3\gamma R^2/4} \right) \right) =: M\tilde{F}(R).
\end{align*}
Again, 
\[ \tilde{F}(R) \leq C' \exp \left( -\delta' \langle R \rangle \right), \] with
\[ \delta' = \frac{4\alpha \beta \sqrt{3\gamma}}{2\alpha + \beta \sqrt{ 3\gamma}}. \] Therefore,
\[ |a_2[\mu](R)|>1/2, \] when \[ \langle R \rangle \geq \frac{1}{\delta'} \log \left( 2C' M \right). \]
When $M$ is large, we can set
\[ z_1 = \log(4MC')/\delta \text{ and } z_2 = \log(4MC')/\delta' .\] 
\end{proof} 
We now move on to the proof of proposition \ref{prop:preservedset}.
\begin{proof}[Proof of Proposition \ref{prop:preservedset}]
We begin by recalling part of the result of Lemma \ref{lem:Gbound} that
\[ \intR G(\mu) \phi_{\beta} \, \dr \dR \leq \frac{2}{\Lambda}\left(\sqrt{\gamma/3}c + 2\sigma^2 \frac{(1+\beta)}{\gamma} + \frac{\Lambda}{2} \right) \sup_{|(\rho', R')|\leq r(\mu)}\phi_{\beta}(\rho', R'). \]
When $M$ is sufficiently large, $z_1,z_2$ limit the choice of $r(\mu)$. Let us work in the case $2/\gamma \leq 1$. We therefore are interested in
\begin{align*} &\sup_{|\rho| < z_1, |R| < \max\{1, 2/\gamma \} z_2} \phi_{\beta}(\rho, R) \leq \exp \left( \beta \sqrt{1 + \frac{4}{\gamma}z^2_1 + 2z_1 z_2 + \gamma z_2^2} \right)\\
& = \exp \left( \sqrt{\beta^2 + \log\left(\frac{2C'M}{c} \right)^2 \left( \frac{4 (\alpha \sqrt{\gamma} + \beta \sqrt{3})^2}{12 \gamma \alpha^2 } + 2 \frac{(\alpha \sqrt{\gamma} + \beta \sqrt{3})(2\alpha +\beta \sqrt{3\gamma})}{24 \alpha^2 \sqrt{\gamma}} + \gamma \frac{(2\alpha + \beta \sqrt{3 \gamma})^2}{48\alpha^2 \gamma}\right)}\right)\\
& = \exp \left(\sqrt{\log\left(\frac{2C'M}{c} \right)^2 \frac{7}{12} + o(\beta)} \right) \\
& \leq CM^\eta,
\end{align*}
where $\eta = \sqrt{7/12} + o(\beta)$. We then choose $\beta$ so that $\eta <1.$ Therefore, we have for some constant $C$,
\[ \intR \mu \phi_{\beta} \, \dr \dR \leq M \Rightarrow \intR G(\mu) \phi_{\beta} \, \dr \dR \leq CM^\eta. \] Therefore, if $M$ is large enough we will map the set $\mathscr{C}_{M,\beta}$ onto itself.
\end{proof}

\subsubsection{Continuity of the function $G$}
In this section we prove the following proposition.
\begin{prop} \label{prop:gcontinuous}
The function $G$ from $\mathcal{P}_\beta$ to $\mathcal{P}_\beta$ is continuous with respect to the topology induced by the norm $\|\cdot\|_\beta,$ and in the topology of weak convergence of measures.
\end{prop}

 First let us describe the main intuition. We start by showing that we can turn the question of continuity of $G$ into a question of continuity of the semigroup $S_t^\mu$ with respect to $\mu$. We will see below that it is straightforward to show that if $\mu_1$ and $\mu_2$ are close in total variation then we can show that $S_t^{\mu_1} \nu$ and $S_t^{\mu_2} \nu$ are close in Wasserstein-1 distance by a stability estimate on the SDEs. As $S_t^\mu$ is a regularising semi-group for every $\mu$, and depends continuously on $\mu$, we expect Wasserstein closeness of $S_t^{\mu_1},S_t^{\mu_2}$ to imply closeness in the $\|\cdot\|_*$ norm. 
We would like to prove an inequality like
\[ W_1(\nu_1, \nu_2) \leq \|\nu_1-\nu_2\|_\beta^a H(\nu_1, \nu_2), \] where $H(\nu_1, \nu_2)$ is a quantity that depends on some norm of $\nabla \nu_1, \nabla \nu_2$. We were not able to prove such an inequality, wich is why we use moments to control the tail behaviour.\\

We start by stating a sequence of lemmas to control and relate the different distances to each other.
\begin{lem}
There exists a function $C(t)$ which is finite for $t$ sufficiently large so that
\[ \| G(\mu_1) - G(\mu_2)\|_\beta \leq C(t) \| S_t^{\mu_1} - S_t^{\mu_2}\|_\beta'. \]
\end{lem}
\begin{proof}
We use the fact that for any $t$ we have $S_t^\mu G(\mu) = G(\mu).$ We also have from the Harris's theorem result that there exists $\lambda, D$ which depend on $\mu_1$ such that
\[ \| S_t^{\mu_1}(f-g)\|_\beta \leq De^{-\lambda t} \|f-g\|_\beta.  \] Using these two facts we have
\begin{align*}
    \|G(\mu_1) - G(\mu_2)\|_\beta & = \|S_t^{\mu_1}G(\mu_1) - S_t^{\mu_2}G(\mu_2) \|_\beta\\
    & \leq \|S_t^{\mu_1}(G(\mu_1) - G(\mu_2))\|_\beta + \|(S_t^{\mu_1} - S_t^{\mu_2})G(\mu_2) \|_\beta \\
   &  \leq De^{-\lambda t} \|G(\mu_1) - G(\mu_2)\|_\beta + \|(S_t^{\mu_1} - S_t^{\mu_2})G(\mu_2)\|_\beta.
\end{align*}
Rearranging this we have
\[ (1-De^{-\lambda t}) \|G(\mu_1) - G(\mu_2) \|_\beta \leq \|(S_t^{\mu_1} - S_t^{\mu_2})G(\mu_2) \|_\beta.  \]
\end{proof}
Next we show that in Wasserstein distances, we can control the difference between $S_t^{\mu_1}$ and $S_t^{\mu_2}$.
\begin{lem}
There exists a constant $C>0$ such that for any $\nu$ we have
\[ W_1(S_t^{\mu_1} \nu, S_t^{\mu_2} \nu) \leq C e^{Ct} \left( \| a_1[\mu_1] - a_1[\mu_2]\|_\infty + \| a_2[\mu_1] - a_2[\mu_2] \|_\infty \right) \]
\end{lem}
\begin{proof}
We show this by construction of an explicit coupling of $S_t^{\mu_1} \nu$ and $S_t^{\mu_2} \nu$. Let $(R_0, \rho_0)$ be distributed with law $\nu$ and construct two stochastic processes $(R^{(1)}_t, \rho^{(1)}_t)$ and $(R^{(2)}_t, \rho^{(2)}_t)$ as solutions to the SDEs
\begin{align*}
    \mathrm{d}R^{(i)}_t &= a[\mu_i](R^{(i)}_t, \rho^{(i)}_t) \,\mathrm{d}t,\\
    \mathrm{d}\rho^{(i)}_t &= - \gamma a_1[\mu_i](\rho^{(i)}_t)\, \mathrm{d}t + \sigma \,\mathrm{d}W_t,
\end{align*}
where both SDEs have the same initial data $(R_0, \rho_0)$ and the same Brownian motion $(W_t)_{t \geq 0}.$ Then the law of $((R^{(1)}_t, \rho^{(1)}_t), (R^{(2)}_t, \rho^{(2)}_t))$ defines a coupling of $S_t^{\mu_1} \nu$ and $S_t^{\mu_2} \nu$. We can also compute 
\begin{align*}
    \mathrm{d} \left( |R^{(1)}_t - R^{(2)}_t| + |\rho^{(1)}_t - \rho^{(2)}_t| \right) & \leq \sgn(R^{(1)}_t - R^{(2)}_t)\left( a[\mu_1](R^{(1)}_t, \rho^{(1)}_t) - a[\mu_2](R^{(2)}_t, \rho^{(2)}_t)\right) \mathrm{d}t\\ & \quad - \gamma  \sgn(\rho^{(1)}_t - \rho^{(2)}_t)\left( a_1[\mu_1](R^{(1)}_t, \rho^{(1)}_t) - a_1[\mu_2](R^{(2)}_t, \rho^{(2)}_t)\right)\mathrm{d}t\\
    & \leq C \left(\|a'[\mu_1]\|_\infty + \|a_1'[\mu_1]\|_\infty \right)\left( |R^{(1)}_t - R^{(2)}_t| + |\rho^{(1)}_t - \rho^{(2)}_t| \right)\mathrm{d}t \\
    & \quad + C \left(\| a_1[\mu_1] - a_1[\mu_2]\|_\infty + \| a_2[\mu_1] - a_2[\mu_2] \|_\infty \right)\mathrm{d}t.
\end{align*}
Integrating this and using the fact that the initial conditions are the same we get
\[ \left( |R^{(1)}_t - R^{(2)}_t| + |\rho^{(1)}_t - \rho^{(2)}_t| \right) \leq C e^{Ct}\left(\| a_1[\mu_1] - a_1[\mu_2]\|_\infty + \| a_2[\mu_1] - a_2[\mu_2] \|_\infty \right). \] Taking expectations then gives our result.
\end{proof}
We now need to prove a sequence of lemmas relating our distances to each other. 
\begin{lem} \label{lem:interpolation}
For any $Z$, we have that
\[ \| f - g\|_* \leq \left(\int (f+g)\phi^{1+a} \right)\sup_{|(\rho, R)|> Z} \phi(\rho, R)^{-a} + \left(\int_{|(\rho, R)|<Z} |f-g|\right)\sup_{|(\rho, R)|<Z}\phi(\rho,R).   \]
\end{lem}
\begin{proof}
 For every $Z$ we have
\[ \| f-g\|_* = \int_{|(\rho,R)| \leq Z} |f-g| \phi + \int_{(\rho,R)| > Z} |f-g| \phi. \] We then bound the first term by
\[ \left(\int_{|(\rho, R)| < Z} |f-g|\right) \sup_{|(\rho, R)| < Z} \phi(\rho, R). \] We then bound the second term by
\[ \left( \int_{|(\rho, R)| \geq Z} |f-g| \phi^{1+a} \phi^{-a} \right) \leq \left( \int (f+g)\phi^{1+a}\right) \sup_{|(\rho, R)| > Z} \phi^{-a}.\]
\end{proof}
\begin{lem}
For any $Z$ we have some constant $C>0$ such that
\[ \int_{|(\rho,R)| \leq Z} |f-g| \, \dr \dR  \leq C\left( \\\int_{|(\rho,R)|\leq Z}(|\nabla f|^2 + |\nabla g|^2) \, \dr \dR \right)^{1/4} \left( W_1(f,g) \right)^{1/2} \]
\end{lem}
\begin{proof}
Let $\eta$ be a smooth mollifying function. We have that
\[ \int_{|(\rho,R)| \leq Z} |f-g|\, \dr \dR = \sup_{\|\psi\|_\infty \leq 1} \int _{|(\rho,R)| \leq Z} (f-g)\psi \, \dr \dR.  \] We also have,
\[ W_1(f,g) = \sup_{\|\psi\|_{Lip} \leq 1} \int (f-g)\psi \, \dr \dR  \] Using these formulations, let us take $\psi$ to have $\|\psi\|_\infty \leq 1$. Furthermore, letting $\eta^\epsilon = \eta(  z/\epsilon)/\epsilon^2$, we can see
\begin{align*}
    &\int_{|(\rho,R)| \leq Z} (f-g) \psi\, \dr \dR \\ =& \int_{|(\rho,R)| \leq Z} f (\psi - \eta^\epsilon * \psi) \, \dr \dR+ \int_{|(\rho,R)| \leq Z} (f-g) \eta^\epsilon * \psi \, \dr \dR + \int_{|(\rho,R)| \leq Z} g( \psi - \eta^\epsilon * \psi) \, \dr \dR\\
     \leq& \|\eta^\epsilon * \psi\|_{Lip} W_1(f,g)  + \int_{|(\rho,R)| \leq Z} \psi \left(f - \eta^\epsilon * f + g - \eta^ \epsilon * g \right) \, \dr \dR\\
     \leq &\|\eta^\epsilon * \psi\|_{Lip} W_1(f,g)  + \int_{z \leq Z} \int_{z'} \psi(z) \eta^{\epsilon}(z') \left(f(z) + g(z) - f(z-z') - g(z-z') \right)\, \dr \dR\\
     \leq &\|\eta^\epsilon * \psi\|_{Lip} W_1(f,g) + \|\eta^\epsilon(z)|z|\|_2 \left(\int_{z\leq Z}\left(|\nabla f|^2 + |\nabla g|^2 \right)\, \dr \dR\right)^{1/2}\\
     \leq & C \epsilon^{-1} W_1(f,g) + C \epsilon \left(\int_{z \leq Z} \left( |\nabla f|^2 + |\nabla g|^2\right)\, \dr \dR\right)^{1/2}.
\end{align*} Optimising over $\epsilon$ gives
\[ \int_{|(\rho,R)| \leq Z} (f-g) \psi \, \dr \dR \leq C W_1(f,g)^{1/2}\left(\int_{z \leq Z} \left( |\nabla f|^2 + |\nabla g|^2\right)\, \dr \dR\right)^{1/4}.\]  
This implies our result.
\end{proof}
Combining these two lemmas we have the following lemma.
\begin{lem}
For any radially decreasing weight $w(z)$ and any $a>0$ we have,
\[ \| f - g\|_* \leq \left(\int (f+g)\phi^{1+a} \right)\sup_{|(\rho, R)|> Z} \phi(\rho, R)^{-a} + CW_1(f,g)^{1/2}\left( \int w(z)(|\nabla f|^2 +|\nabla g|^2)\right)^{1/4}\sup_{|(\rho, R)|<Z}\phi(\rho,R).   \] This implies
\[ \|f-g\|_\beta \leq C \|f+g\|_{\beta(1+a)}^{1/(1+a)}W_1(f,g)^{a/2(1+a)}\left( \int w(z)\left( |\nabla f|^2 + |\nabla g|^2\right)\right)^{a/4(1+a)}. \]
\end{lem}

This section is inspired partly from \cite{Cao}. We extend the regularisation estimates from this paper by including the weight $\partial_\rho a_1[\mu](\rho)$ in front of the terms $\partial_R f$, this allows us to deal with the weak mixing term.

\begin{lem}
For any weight function $m$, with $|\nabla m| \leq Cm$, $\partial^2_{\rho^2}m \leq Cm$, we have
\begin{align*}
    \frac{\mathrm{d}}{\mathrm{d}t} \intR m f^2 \, \dr \dR &\leq C\intR m f^2 \, \dr \dR - \sigma^2 \intR m (\partial_\rho f)^2 \, \dr \dR,\\
    \frac{\mathrm{d}}{\mathrm{d}t}\intR m (\partial_\rho f)^2 \, \dr \dR &\leq C \intR m f^2 \dr \dR + C\intR  m (\partial_\rho f)^2 \, \dr \dR - \intR m \partial_\rho a_1 \partial_\rho f \partial_R f \, \dr \dR\\
    & \quad - \sigma^2 \intR m (\partial^2_{\rho^2} f)^2 \, \dr\dR,\\
    \frac{\mathrm{d}}{\mathrm{d}t} \intR m \partial_\rho f \partial_R f \dr \dR & \leq C \intR m f^2 \, \dr \dR + C\int m|\partial_\rho f \partial_R f| \, \dr \dR - \intR m \partial_\rho a_1 (\partial_R f)^2 \, \dr \dR \\
    & \quad - \sigma^2 \intR m(\partial^2_{\rho^2} f) (\partial^2_{\rho,R}f) \, \dr \dR,\\
    \frac{\mathrm{d}}{\mathrm{d}t} \intR m (\partial_R f)^2 \, \dr \dR &\leq C \intR m f^2 \dr \dR + C\intR m(\partial_R f)^2 \, \dr \dR - \sigma^2 \intR m (\partial^2_{\rho R}f)^2 \, \dr \dR.
\end{align*}
\end{lem}
\begin{proof}
These inequalities are all essentially calculations. Let us start with the first one,
\begin{align*}
    \frac{\mathrm{d}}{\mathrm{d}t} \intR m f^2 \, \dr \dR &= 2\intR mf \left(- \partial_R((a_1-a_2)f) + \frac{\sigma^2}{2} \partial_{\rho^2}^2 f + \gamma \partial_\rho(a_1 f) \right) \dr \dR\\
    & = 2 \intR \left((a_1-a_2)\partial_R(mf)f - \frac{\sigma^2}{2} \partial_\rho f \partial_\rho (mf) - \gamma a_1 f \partial_\rho (mf)\right) \dr \dR\\
    & = 2 \intR \left((a_1-a_2)\partial_R m - \gamma a_1 \partial_\rho m \right)f^2 \,\dr \dR \\
    & \quad + \intR \left((a_1-a_2)m \partial_R(f^2) - \gamma a_1 m \partial_\rho (f^2) \right) \dr \dR \\
    & \quad - \intR \frac{\sigma^2}{2} \partial_\rho m \partial_\rho (f^2) \dr \dR  - \sigma^2 \intR m (\partial_\rho f)^2 \,\dr \dR\\
    &= \intR \left((a_1-a_2)\partial_R m + \partial_R a_2 m - \gamma a_1 \partial_\rho m + \gamma \partial_\rho a_1 m + \frac{\sigma^2}{2} \partial^2_{\rho^2}m \right)f^2\, \dr \dR\\
    & \quad - \sigma^2 \intR m (\partial_\rho f)^2 \,\dr \dR,
\end{align*}
\begin{align*}
    \frac{\mathrm{d}}{\mathrm{d}t}\intR m (\partial_\rho f)^2 \, \dr \dR &= 2 \intR m \partial_\rho f \partial_\rho\left(-\partial_R((a_1-a_2)f) + \frac{\sigma^2}{2} \partial^2_{\rho^2}f + \gamma \partial_\rho(a_1 f) \right)\dr \dR\\
    &= 2 \intR m \partial_\rho f \left(-\partial_\rho a_1 \partial_R f + \partial_R a_2 \partial_\rho f +(a_2-a_1)\partial^2_{\rho R}f \right)\dr \dR\\
    & \quad + 2\intR m \partial_\rho f \left(\frac{\sigma^2}{2} \partial^3_{\rho^3} f + \gamma \partial^2_{\rho^2}a_1 f + 2 \gamma \partial_\rho a_1 \partial_\rho f + \gamma a_1 \partial^2_{\rho^2}f \right)\dr \dR\\
    &= -2 \intR m \partial_\rho a_1 (\partial_\rho f)(\partial_R f)\dr \dR + \intR \left(m \partial_R a_2 +(a_1-a_2)\partial_R m) \right)(\partial_\rho f)^2\,\dr \dR\\
    & \quad - \sigma^2 \intR m (\partial^2_{\rho^2} f)^2 \,\dr \dR + \frac{\sigma^2}{2} \intR (\partial^2_{\rho^2}) (\partial_\rho f)^2 - \gamma \intR \partial_\rho(m \partial^2_{\rho^2}a_1) f^2 \, \dr \dR \\
    & \quad + \intR (4 \gamma \partial_\rho a_1 m - \gamma \partial_\rho(a_1 m) ) (\partial_\rho f)^2 \,\dr \dR\\
    & = - 2 \intR m \partial_\rho a_1 \partial_\rho f \partial_R f \,\dr \dR - \sigma^2 \intR m(\partial^2_{\rho^2}f)^2 \, \dr \dR - \gamma \intR \partial_\rho(m \partial^2_{\rho^2} a_1) f^2 \, \dr \dR\\
    & \quad + \intR \left(m (\partial_R a_2 + 3\gamma \partial_\rho a_1) + (a_1-a_2)\partial_R m \dr \dR- \gamma a_1 \partial_\rho m + \frac{\sigma^2}{2}\partial^2_{\rho^2} m \right)(\partial_\rho f)^2\,\dr \dR,
\end{align*}
\begin{align*}
    \frac{\mathrm{d}}{\mathrm{d}t} \intR m \partial_\rho f \partial_R f \, \dr \dR & = \intR m \partial_\rho f \partial_R \left(-\partial_R ((a_1-a_2)f) + \frac{\sigma^2}{2} \partial^2_{\rho^2}f + \gamma \partial_\rho (a_1 f) \right) \dr \dR\\
    & \quad + \intR m \partial_R f \partial_\rho \left(-\partial_R((a_1-a_2)f) + \frac{\sigma^2}{2} \partial^2_{\rho^2} f + \gamma \partial_\rho (a_1 f) \right)\dr \dR\\
    &=\intR m \left(\partial^2_{R^2}a_2 f \partial_\rho f + 3 \partial_Ra_2 \partial_R f \partial_\rho f - \partial_\rho a_1 (\partial_R f)^2 - (a_1-a_2)(\partial_\rho f \partial^2_{R^2}f + \partial_R f \partial^2_{\rho R}f)  \right) \, \dr \dR\\
    & \quad + \frac{\sigma^2}{2} \intR m \left( \partial_\rho f \partial^2_{\rho^2 R}f + \partial_R f \partial^3_{\rho^3} f\right) \, \dr \dR\\
    & \quad + \gamma \intR m \left(\partial_\rho a_1 \partial_\rho f \partial_R f + a_1 \partial_\rho f \partial^2_{\rho R} f + \partial^2_{\rho^2}a_1 f \partial_R f + 2 \partial_\rho a_1 \partial_\rho f \partial_R f +a_1 \partial_R f \partial_{\rho^2} f \right)\,\dr \dR\\
    &=  \frac{1}{2}\intR \left(\partial_{R^2}^2 a_2 \partial_\rho m \dr \dR + \partial_{\rho^2}^2 a_1 \partial_R m \right)f^2 \, \dr \dR  \\
    &\quad  - \intR m \partial_\rho a_1 (\partial_R f)^2 \dr \dR - \sigma^2 \intR m \partial^2_{\rho^2}f \partial^2_{\rho R} f \, \dr \dR\\
    &\quad + \intR  \left(2m \partial_R a_2 + (a_1-a_2) \partial_R m + \frac{\sigma^2}{2} \partial^2_{\rho^2}m + 2 \gamma \partial_\rho a_1 m - \gamma a_1 \partial_\rho m\right)\partial_\rho f \partial_R f  \, \dr \dR,
\end{align*}
\begin{align*}
    \frac{\mathrm{d}}{\mathrm{d}t} \intR m (\partial_R f)^2 \dr \dR & = 2 \intR m \partial_R f \partial_R \left( - \partial_R ((a_1-a_2)f) + \frac{\sigma^2}{2} \partial^2_{\rho^2} f + \gamma \partial_\rho(a_1 f)\right)\, \dr \dR \\
    &= 2 \intR m (\partial_R f) \left( \partial^2_{R^2} a_2 f + 2 \partial_R a_2 \partial_R f - (a_1-a_2) \partial^2_{R^2}f \right)\, \dr \dR \\
    & \quad + \sigma^2 \intR m \partial_R f \partial^3_{R\rho^2}f \, \dr \dR + 2 \gamma \intR m \partial_R f (\partial_\rho a_1 \partial_R f + a_1 \partial^2_{\rho R} f )\, \dr \dR \\
    &= - \intR \partial_R(m \partial^2_{R^2}a_2)f^2 m \, \dr \dR - \sigma^2 \intR m (\partial^2_{\rho R} f)^2 \, \dr \dR \\
    & \quad  \intR \left(3m \partial_R a_2 + (a_1 -a_2) \partial_R m + \frac{\sigma^2}{2} \partial^2_{\rho^2} m + \gamma m \partial_\rho a_1 - \gamma a_1 \partial_\rho m \right)(\partial_R f)^2 \, \dr \dR.
\end{align*}
\end{proof}

Using this lemma we can prove a first regularisation result.
\begin{lem}
Suppose that $m$ satisfies the conditions from the previous lemma, then we can choose $A_1, A_2, A_3$ such that $\mathcal{F}(t,f)$ defined by
\begin{align*} \mathcal{F}(t,f):= \intR\left[m \left(f^2 + A_1 t^2 (\partial_\rho f)^2 + A_2 t^4 (\partial_\rho a_1) (\partial_\rho f)(\partial_R f) + A_3 t^6 (\partial\rho a_1)^2(\partial_R f)^2 \right)\right] \, \dr \dR, 
\end{align*}
is a decreasing quantity for $t$ sufficiently small. Specifically, that there exists a $t_*$ such that for $t \leq t_*$ there is some $\Lambda$ with 
\[ \frac{\mathrm{d}}{\mathrm{d}t} \mathcal{F}(t,f_t) \leq \intR f^2 m \, \dr \dR - \Lambda \intR m \left((\partial_\rho f)^2 \, \dr \dR + t^4 (\partial_\rho a_1)^2 (\partial_R f)^2 \right) \, \dr \dR. \]
\end{lem}
\begin{proof}
We can differentiate $\mathcal{F}$ using the identities from the previous lemma. This gives 
\begin{align*}
    \frac{\mathrm{d}}{\mathrm{d}t}\mathcal{F}(t,f_t) &\leq C(1+A_1t^2+A_2t^4+A_3t^6)\intR f^2 m \, \dr \dR\\
    & \quad + \intR m(2A_1t  + A_1t^2C- \sigma^2 )(\partial_\rho f)^2 \, \dr \dR\\
    & \quad + \intR m(\partial_\rho a_1)(4A_2t^3 + A_2 t^4 C-A_1 t^2)|(\partial_\rho f)(\partial_R f)| \, \dr \dR\\
    & \quad + \intR m (\partial_\rho a_1)^2(6A_3 t^5  + CA_3 t^6  - A_2t^4)(\partial_R f)^2 \, \dr \dR\\
    & \quad -\sigma^2 \intR m \left(A_1t^2(\partial^2_{\rho^2}f)^2 +A_2t^4 \partial_\rho a_1 (\partial^2_{\rho^2}f)(\partial^2_{\rho R}f) + A_3t^6 (\partial_\rho a_1)^2 (\partial^2_{\rho R}f)^2 \right)\, \dr \dR .
\end{align*}
If $A_2 \leq 4A_1 A_3$, then the last line will vanish.
Splitting the third line up with Young's inequality will give
\begin{align*}
    \frac{\mathrm{d}}{\mathrm{d}t}\mathcal{F}(t,f_t) & \leq C(1+A_1 t^2 + A_2 t^4 + A_3 t^6) \intR f^2 m \, \dr \dR\\
    & \quad + \intR m \left(2A_1 t + A_1 t^2 C - \sigma^2 + \frac{1}{2\eta}\left|4A_2t + A_2 t^2C -A_1 \right| \right)(\partial_\rho f)^2 \, \dr \dR\\
    & \quad + \intR m \left(6A_3 t^5 + CA_3 t^6 - A_2t^4- + \frac{\eta}{2} \left|4A_tt^5 + A_2 C t^6 - A_1 t^4 \right| \right)(\partial_\rho a_1)^2(\partial_R f)^2 \, \dr \dR
\end{align*}
Now if we choose $\eta = A_1/A_2$ we have
\begin{align*} \frac{\mathrm{d}}{\mathrm{d}t}\mathcal{F}(t,f_t) &\leq C(t) \intR f^2 m  \, \dr \dR- ( \sigma^2 - A_1^2/2A_2 + o(t))\intR m (\partial_\rho f)^2 \, \dr \dR \\
&\quad - \left(\frac{A_2}{2}t^4+o(t^5) \right)\intR (\partial_\rho a_1)^2 (\partial_R f)^2 \, \dr \dR.  
\end{align*}
Therefore, if we set $A_1 =A_2=A_3 = \sigma^2$ we will have
\begin{align*}
\frac{\mathrm{d}}{\mathrm{d}t} \mathcal{F}(t,f_t) &\leq C(1+A_1t^2 + A_2 t^4 + A_3 t^6) \intR f^2m \, \dr \dR - (\sigma^2/2 + o(t)) \intR m (\partial_\rho f)^2 \, \dr \dR\\
& \quad - (\sigma^2/2 + o(t))t^3 \intR m (\partial_\rho a_1)^2 (\partial_R f)^2\, \dr \dR. 
\end{align*}
\end{proof}

\begin{proof}[Proof of proposition \ref{prop:gcontinuous}]
We know that 
\[ \intR f^2 m \, \dr \dR= \intR |f \phi| |fm \phi^{-1}| \, \dr \dR \leq \intR |f| \phi \|fm \phi^{-1}\|_\infty\, \dr \dR. \] Now, if $\mu \in \mathcal{C}_{M,\beta}$ then $G(\mu)$ is a $C^\infty$ density by H\"ormander's theorem; and, in particular, it is in $L^\infty$ and since $m\phi^{-1}$ is in $L^\infty$ this means that $\|G(\mu)m\phi^{-1}\|_\infty < \infty$, though we can't get uniform estimates on this over $\mathcal{C}_{M,\beta}$.

Integrating the result of the previous lemma gives
\[ \mathcal{F}(t,f_t) \leq C(t) \|f_t\sqrt{m}\|_2^2. \]
This implies that 
\[  \frac{\mathrm{d}}{\mathrm{d}t} \|f_t \sqrt{m}\|^2_2 \leq C \|f_t \sqrt{m}\|_2^2.\] 
Furthermore, if $\mu_1, \mu_2 \in \mathcal{C}_{M,a(1+\beta)}$ then 
\begin{align*}
    &\|S_t^{\mu_1}G(\mu_2) -S_t^{\mu_2}G(\mu_2)\|_\beta \\ 
    &\quad \leq C\|S_t^{\mu_1} G(\mu_2) + S_t^{\mu_1}G(\mu_2)\|^{1/(1+a)}  W_1(S_t^{\mu_1}G(\mu_2),S_t^{\mu_1}G(\mu_2))^{a/2(1+a)}\|m (\partial_\rho a_1)^2 (\nabla f + \nabla g)\|_2^{a/4(1+a)}\\
    & \quad \leq C(t)M^{1/(1+a)}\left( \|a_1[\mu_1]-a_1[\mu_2]\|_\infty + \|a_2[\mu_1]-a_2[\mu_2]\|_\infty \right)^{a/2(1+a)}\|G(\mu_2)m\|_2^{a/2(1+a)}
\end{align*}
We know that $\|G(\mu_2)m\|_2 < \infty$ as 
\[ \|G(\mu_2)m\|_2^2 \leq \|G(\mu_2)\phi_\beta\|_1\|G(\mu_2) m \phi_\beta^{-1}\|_\infty, \] and we know the second term is finite as $G(\mu_2)$ is a continuous probability density so is in $L^\infty.$
We also have that
\[ |a_1[\mu_1](\rho)-a_1[\mu_2](\rho)| \leq \intR |b(\rho - \rho')| |\mu_1(\rho,R) - \mu_2(\rho,R)|\, \mathrm{d}\rho \mathrm{d}R \leq \|b\|_\infty \|\mu_1 - \mu_2\|_{TV} \leq C \|\mu_1 - \mu_2\|_\beta. \]

Therefore, 
\[ \|S_t^{\mu_1}G(\mu_2) - S_t^{\mu_2}G(\mu_2)\|_\beta \leq C(t) \|\mu_1-\mu_2\|_\beta^{a/1(1+a)} \|G(\mu_2)m\|_2^{a/1(1+a)}. \]
Hence,
\[ \|G(\mu_1) - G(\mu_2)\|_\beta \leq C(t) \|\mu_1-\mu_2\|_\beta^{a/2(1+a)} \|G(\mu_2)m\|_2^{a/2(1+a)}.  \] This gives strong continuity of the map $G$ from $\mathcal{C}_{M, (1+a)\beta}$ to itself.

We also have that if $\mu_n \rightarrow \mu_\infty$ weakly then $a_i[\mu_n](\rho) \rightarrow a_i[\mu_\infty](\rho),$ so we also have that in this case
\[ \|G(\mu_n) - G(\mu_\infty)\|_\beta \rightarrow 0. \] Which implies convergence in total variation and hence that $G(\mu_n)$ converges weakly towards $G(\mu_\infty)$.

\end{proof}

\subsubsection{Putting the steps together}
We conclude by gathering the results of the previous section to prove our main Theorem~\ref{thm:main}. First we state Schauder fixed point theorem as can be found in \cite[Theorem 2.3.7]{Smar74}. 

\begin{theorem}[Schauder Fixed Point Theorem] \label{Theorem: Schauder}
Let $S$ be a non-empty, convex closed subset of a Hausdorff topological vector space and $F$ a mapping of $S$ into itself so that $F(S)$ is compact, then $F$ has a fixed point.
\end{theorem}

\begin{proof}[Proof of Theorem \ref{thm:main}]
We apply Schauder's fixed point theorem to the map $G$, the set $\mathscr{C}_{M,\beta}$ with $M$ large enough and $\beta$ small enough so that Propositions \ref{prop:Gdefn}, \ref{prop:preservedset} and \ref{prop:gcontinuous} are valid. We work in the topology of weak convergence of measure so that the set $\mathscr{C}_{M,\beta}$ is clearly convex and is compact thanks to Lemma \ref{lem:compact}. The function $G$ is well-defined thanks to Proposition \ref{prop:Gdefn}, continuous thanks to Proposition \ref{prop:gcontinuous}. This allows us to verify all the conditions of Schauder's fixed point theorem. Lastly we note that as we have shown that $G(\mu)$ has $\|G(\mu)\|_\beta \leq M$ (Proposition \ref{prop:Gdefn}) and is smooth thanks to Malliavin's theorem then this properties are also true for our steady state.
\end{proof}

\section{Conclusion and future research directions}\label{s:outlook}

In this paper we investigated the existence of steady states to a nonlinear Fokker-Planck equation, describing the evolution of player ratings competing in zero sum games. The existence result is based on Schauder's fixed point theorem and investigating the behaviour of a corresponding linear problem using hypercoercivity techniques. To this point we are not able to prove uniqueness or say anything about the trend to equilibrium. Hence the most natural next steps in understanding the long-time behaviour of this equation are:
\begin{itemize}
    \item \textbf{Uniqueness of the steady state}. It seems likely that for the equation studied here the steady state will be unique. However, it is challenging to prove. The kinetic structure means that we cannot view the steady state as the minimiser of a convex energy. The non-explicitness for the function $G(\mu)$ makes it challenging to work with. Furthermore, it is difficult to see how we could convert our fixed point argument into a contraction mapping argument.
    \item \textbf{Linear and non-linear stability of the steady state.} The linear stability of the steady state found in this paper seems to be a much more tractable problem. The numerical results presented in Section~\ref{s:numerics} suggest that the solution to the non-linear equation is converging exponentially fast towards its steady state in spaces weighted by $\phi_\beta.$ This is also close to a typical situation for the application of hypocoercivity theory. However, significant challenges remain. These are essentially the same as the those for the linear equation \eqref{eq:linearmain}. We are not able to use Harris's theorem for the linearised equation as Harris's theorem requires the preservation of positivity.
\end{itemize}

\section*{Acknowledgements} JE acknowledges partial support from the Leverhulme Trust, Grant ECF-2021-134. The author(s) would like to thank the Isaac Newton Institute for Mathematical Sciences, Cambridge, for support and hospitality during the programme \textit{Frontiers in Kinetic Theory: connecting microscopic and macroscopic scales} where work on this paper was undertaken. 
\bibliographystyle{abbrv}
\bibliography{ref}

\end{document}